\documentclass[A4paper,oneside,11pt]{amsart}
\usepackage{amssymb,amsthm,amsmath,marvosym,amsfonts,fontenc}
\usepackage{comment}
\usepackage{graphicx}
\usepackage{hyperref}
\usepackage{enumerate}

\newtheorem{theorem}{Theorem}
\newtheorem*{theorem*}{Theorem}
\newtheorem*{observation*}{Observation}

\newtheorem{subtheorem}{Theorem}[theorem]

\newtheorem{lemma}[theorem]{Lemma}

\newtheorem{observation}[theorem]{Observation}

\newtheorem{proposition}[theorem]{Proposition}

\newtheorem{corollary}[theorem]{Corollary}

\newtheorem{definition}[theorem]{Definition}

\theoremstyle{remark}
\newtheorem{claim}[subtheorem]{Claim}

\def\cA{\mathcal{A}}

\def\cD{\mathcal{D}}
\def\cF{\mathcal{F}}
\def\cB{\mathcal{B}}

\def\cL{\mathcal{L}}
\def\eps{\varepsilon}
\def\epsilon{\varepsilon}
\def\phi{\varphi}
\def\veceps{\boldsymbol{\eps}}

\def\veczero{\boldsymbol{0}}

\usepackage{graphicx,color}
\title{Maximum number of sum-free colorings in finite abelian groups} 

\author{Hi\d{\^e}p H\`an} 
\address{Departamento de Matem\'atica y Ciencia de la Computaci\'on,
Universidad de Santiago de Chile}
\email{han.hiep@gmail.com}
\thanks{The first author was supported by the FONDECYT Iniciaci\'on grant  11150913.}
\author{Andrea Jim\'enez}
\address{CIMFAV, Facultad de Ingenier\'ia, Universidad de Valpara\'iso}
\email{andrea.jimenez@uv.cl}
\thanks{The second author was  supported by  CONICYT/FONDECYT/POSTDOC\-TORADO~3150673.}
\thanks{The authors acknowledge the support of Nucleo Milenio Informaci\'on y Coordinaci\'on en Redes ICM/FIC RC130003, Chile.}

\begin{document}

\maketitle
\begin{abstract}
An $r$-coloring of a subset $A$ of a finite abelian group $G$ is called 
sum-free if it does not induce a  monochromatic Schur triple, i.e., a triple of elements $a,b,c\in A$ with $a+b=c$.
We investigate $\kappa_{r,G}$, the maximum number of sum-free  $r$-colorings admitted by subsets of  $G$, and our results show a close relationship between 
$\kappa_{r,G}$ and largest sum-free sets of $G$.

Given a sufficiently large abelian group $G$ of type $I$, i.e., $|G|$ has a prime divisor $q$ with $q\equiv 2\pmod 3$.
For $r=2,3$ we show that  a subset $A\subset G$ achieves $\kappa_{r,G}$ if and only if $A$ is a largest sum-free set of $G$.
For even order $G$ the result extends to $r=4,5$, where the phenomenon  persists only if $G$ has a unique largest sum-free set.
On the contrary, if the  largest sum-free set in $G$ is not unique then
$A$ attains $\kappa_{r,G}$ if and only if it is the union of two largest sum-free sets (in case $r=4$)
and the union of three (``independent'') largest sum-free sets (in case $r=5$).

Our approach relies on the so called container method and can be extended to larger $r$ in case $G$ is of even order and contains sufficiently many largest sum-free sets.
\end{abstract}

\section{Introduction}

A \emph{Schur triple} in an abelian group $G$ is a triple $(a,b,c)$ with $a+b=c$, and  
a set $A\subset G$ is  \emph{sum-free} if $A$ contains no such triple. 
Given a not necessarily sum-free set $A\subset G$,  a coloring of the elements of $A$ with $r$ colors  is called a 
\emph{sum-free $r$-coloring} if each of the color classes  is a sum-free set. Sum-free colorings are among the classical objects studied in extremal combinatorics 
and can be traced back  to Schur's theorem~\cite{Schur}, one of the first results in Ramsey theory.

In this paper we investigate  the maximum number of sum-free colorings admitted by  subsets of a given finite abelian group. 
This is a variant of a problem posed by Erd\H{o}s and Rothchild~\cite{Erdosproblem1, Erdosproblem2} for graphs, see Section~\ref{sec:related}.
Let $\kappa_r(A)$   denote the number of all sum-free $r$-colorings of $A\subset G$ and let  the maximum over all $A\subset G$ be denoted by \[\kappa_{r,G}=\max\{\kappa_r(A)\colon A\subset G\}.\]
We are interested in the questions as how large $\kappa_{r,G}$ can be, given $r\geq 2$ and $G$, and which subsets of $G$ achieve the maximum.

A straightforward lower bound for $\kappa_{r,G}$ is obtained by 
considering a largest sum-free set  $B\subset G$, which gives $\kappa_{r,G}\geq r^{|B|}$.
The size of largest sum-free sets of $G$, denoted by $\mu(G)$, is a  classical and well-understood quantity which depends only on the factorization of $G$. 
The characterization of $\mu(G)$ distinguishes the following three types.
 \begin{definition}
Let $G$ be a finite abelian group of order $n$.
If $n$ has a prime divisor $q$ such that $q \equiv 2 \pmod 3$,
then we say that $G$ is a type~I group. In addition, $G$ is called type I($q$) if $q$ is the smallest such prime. 
If $G$ is not of type I and $3|n$ then we say that $G$ is of type II. Otherwise $G$ is called a type III group.
\end{definition}
For groups $G$ of type I and type II the quantity $\mu(G)$ was  determined by Diananda and Yap~\cite{DianandaYap}.
 The problem for groups of type III  appears to be far more complicated and was only resolved decades later by Green and Ruzsa~\cite{GreenRuzsa} (see \cite{RhemtullaStreet,Yap1,Yap2} for partial results
 for type III groups and~\cite{BPR} for the characterization of the largest sum-free sets therein). The results in~\cite{DianandaYap,GreenRuzsa} 
 determine $\mu(G)$ as follows: 
 \[\mu(G)=\begin{cases}\left(\frac13+\frac1{3q}\right)n&\text{if $G$ is of type I($q$),}\\
\frac n3 &\text{if $G$ is of type II,}\\
\left(\frac13-\frac1m\right)n&\text{if $G$ is of type III and $m$ is the largest order of an element in $G$.}
 \end{cases}\]
 
\medskip
For arbitrary abelian groups, we have the following upper bounds which are asymptotically sharp in the exponent for $r=2,3$.
\begin{proposition}\label{rem:contr}
Given $r\geq 2$ and a finite abelian group $G$ of order $n$, then
\[\log_2 (\kappa_{2,G})\leq {\mu(G)+ O(n(\log n)^{-\frac 1{45}})},\]
and for  $r\geq 3$
\[\log_3(\kappa_{r,G}) \leq \frac{r \mu(G)}3 + O(n(\log n)^{-\frac 1{45}}).\]
\end{proposition}

Although $\mu(G)$ is known for all finite abelian groups, it is safe to claim that those of type~I are much better understood (see Section~\ref{sec:groups}).
This additional knowledge allows us  to completely resolve the problem  for two and three colors in groups of type I of sufficiently large order. 
In these cases the straightforward lower bound from above is indeed sharp and only the largest sum-free sets achieve the maximum.
\begin{theorem}
\label{thm:main23}
Let $r \in \{2,3\}$, $q \in \mathbb{N}$ and 
let $G$ be a type~I($q$) group of sufficiently large order.
Then $\kappa_{r,G}=r^{\mu(G)}$ and $\kappa_r(A)=\kappa_{r,G}$ if and only if $A$ is a largest sum-free set in~$G$.
\end{theorem}
For more than three colors this phenomenon does not persist in general and the problem becomes considerably more complicated. We therefore restrict our consideration to 
type I(2) groups, i.e., those of even order.
For these groups our second result resolves the problem for $r=4,5$.
Further,  we shall see in the course of the paper that our method can be extended to more than five colors. 
We refer to Section~\ref{sec:concludingremarks} for further discussions.

Before stating the result we  note that two largest sum-free sets $B_1, B_2$ in an  abelian group of even order give rise to  another   
through $B_3=B_1\triangle B_2$ (see Corollary~\ref{remark:reduction}). In particular, $B_3\subset B_1\cup B_2$ holds in this case and there
is no even order group with exactly two largest sum-free sets. To distinguish the two possible cases we call a tuple $(B_1,\dots, B_t)$ 
of largest sum-free sets  \emph{independent}, if none of the 
$B_i$'s is contained in the union of the remaining ones, or \emph{dependent},  if the opposite holds. 

\begin{theorem}\label{thm:main45} 
Let $G$ be a sufficiently large group of  even order. 
 If $G$ contains a unique largest sum-free set then this set and only this set maximizes the number of sum-free $r$-colorings for 
$r=4,5$. Otherwise
\begin{itemize}
\item $A\subset G$ maximizes the number of sum-free $4$-colorings if and only if
$A$ is the union of two  largest sum-free sets.
\item $A\subset G$ maximizes the number of sum-free $5$-colorings if and only if
$A$ is the union of three largest sum-free sets $B_1,B_2, B_3$. 
Moreover, if $(B_1,B_2,B_3)$ is independent, then $\kappa_5(A)=(1+o(1))181440\cdot6^{n/2}$ and if  $(B_1,B_2,B_3)$ is dependent, then $\kappa_5(A)=(1+o(1))90\cdot6^{n/2}$.
\end{itemize}
\end{theorem}
To emphasize the last point of the theorem note that the number of sum-free $5$-colorings admitted by 
an independent  triple and that of a dependent one are  within one another  by a multiplicative 
constant independent of $n$. In contrast, the stability type Theorem~\ref{thm:stability5} implies that any set which differs from these extremal configurations by 
$\Omega\big(n(\log n)^{-1/27}\big)$ elements admits  exponentially fewer 
sum-free $5$-colorings. This phenomenon neither appears for $r=4$ nor for $r=6$ or $r=7$ (see Section~\ref{sec:concludingremarks}).

\subsection{Related results}\label{sec:related}
Problems analogous to the ones considered in this paper were investigated  for many other discrete structures (see, e.g.,
\cite{Yuster,ABKS,PikhurkoYilma, LPRS,LPS,LefmannPerson,HKL}).

Among them the one concerning clique-free edge colorings of graphs is the most prominent one, which moreover seems closest to the problem studied here
due to the relationship between triangles and Schur triples\footnote{In $\mathbb{F}_2^n$, for example, consider for given $A\subset \mathbb{F}_2^n$ the Cayley graph $G_A$ which consists of the vertex set $\mathbb F_2^n$ and in which
$\{a,b\}$ forms an edge if and only if $a+b\in A$. A triangle $a,b,c$ in $G_A$ then corresponds to the Schur triple $a+b$, $b+c$ and $c+a=c-a$ in $A$.}.

This problem was raised by Erd\H{o}s and Rothchild in~\cite{Erdosproblem1} (see also \cite{Erdosproblem2}) and in~\cite{Yuster} Yuster  
showed  that, among all graphs on $n$ vertices, only the largest triangle-free graphs 
maximize the number of triangle-free $2$-colorings.
Using Szemer\'edi's regularity lemma Alon, Balogh, Keevash and Sudakov~\cite{ABKS}   generalized the result to $r=2, 3$ colors and  cliques $K_k$ of size $k\geq 3$. In this case 
 the Tur\'an graphs $T_{k-1}(n)$,  i.e., the balanced complete $(k-1)$-partite graphs on $n$ vertices, are the unique graphs which attain  
 the maximum number of $K_k$-free $r$-colorings.
 
Similar to sum-free colorings this phenomenon  does not persist for $r>3$ and the question becomes significantly harder. 
Building on the work of Alon et al., Pikhurko and Yilma~\cite{PikhurkoYilma} determine the unique maximizers 
for  $(r,k)=(4,3)$ and for $(r,k)=(4,4)$. These turn out to be the Tur\'an graphs
$T_{4}(n)$  for $k=3$ and   $T_{9}(n)$ for $k=4$.
For any other pair $(r,k)$, in particular for $r=5$ and $k=3$,  the problem remains  open.

% Despite the close relationship we do not see how to 
% directly deduce our result in Theorem~\ref{thm:main45} in the case $r=4$ from the result of Pikhurko and Yilma. 
% However, due to this relationship we believe that  number of $K_3$-free $5$-colorings is at most $6^{n^2/4+o(n^2)}$. 
% We hope to address  this question in near future.
 
\section{Outline of the proofs and stability theorems}
One possible approach to our problem is to use Green's regularity lemma for abelian groups~\cite{GreenRL}. 
In various analogous contexts regularity lemmas have proven to be a suitable tool~\cite{ABKS,PikhurkoYilma,LPRS,LPS,LefmannPerson}.
While this may work well here for groups with many subgroups such 
as $\mathbb F_p^n$ the technical difficulties are considerable for those lacking subgroups. 
A novel aspect of our work is  to avoid these difficulties by employing the so-called container method. %~(see~\cite{BMS,SaxtonThomason,GreenRuzsa}).
%, which would also work for settings in which regularity lemmas were used.
For sum-free sets this  comes in the form of a result by Green and Ruzsa~\cite{GreenRuzsa}
 which we  state in a slightly modified form. 
We also note that instead of the results of Green and Ruzsa one could also use the container results of  
Balogh, Morris, Samotij in~\cite{BMS}, that of Saxton, Thomason in~\cite{SaxtonThomason} or a version by Alon et al in~\cite{ABMS} which is closely related to that in~\cite{BMS}. 
 
\begin{theorem}[Proposition~2.1 in~\cite{GreenRuzsa}]\label{thm:container}
Let $G$ be a finite abelian group of sufficiently large order~$n$.
For every subset $A\subset G$ there is a family $\cF=\cF(A)$ of subsets of 
$A$ (called container family of~$A$) with the following properties
\begin{enumerate}
\item \label{it:container1}$\log_2 |\cF|\leq n(\log n)^{-1/18}$;
\item \label{it:container2}Every sum-free set $I\subset A$ is contained in some $F\in\cF$;
\item \label{it:container3}If $F\in\cF$ then $F$ contains at most $n^2(\log n)^{-1/9}$ Schur triples.
\end{enumerate}
\end{theorem}

Roughly speaking the theorem states that all sum-free sets in $A$ can be ``captured'' by a small family of almost sum-free sets. %number of sets each of which contains only few sums. 
The  result of Green and Ruzsa  is stated for $A=G$,
however, we obviously obtain the family $\cF$ as above by taking the intersection of each $F$ with $A$. Here and in what follows 
the dependence on $A$ is regularly suppressed as it is clear from the context.

With Theorem~\ref{thm:container} as the starting point we make the following simple but crucial observation.
\begin{observation}
\label{obs:PhiInContainer}
Let $\cF=\cF(A)$ be  a container family  as in Theorem~\ref{thm:container} and let $\Phi_r(A)$ denote the set of all sum-free $r$-colorings of $A$.
To each $\phi\in\Phi_r(A)$ assign a tuple $(F_1,\dots,F_r) \in \cF^r$  such that $\phi^{-1}(i) \subseteq F_i$ for every $i \in [r]$ and 
let $\Phi(F_1,\dots,F_r)$ denote the set of all $\phi\in\Phi(A)$ assigned to $(F_1,\dots,F_r)$.  
Note that this assignment is possible due to~(\ref{it:container2}) of Theorem~\ref{thm:container} and we have
\begin{align}\label{eq:PhiInContainer1}\Phi_r(A)=\bigcup_{(F_1,\dots,F_r)\in\cF^r}\Phi(F_1,\dots,F_r).\end{align}
Further, let $n_k$ denote the number of elements in $F_1\cup\dots\cup F_r=A$ which are contained in exactly~$k$  sets of  $(F_1,\dots,F_r)$.
Then \begin{align}\label{eq:PhiInContainer2}|\Phi(F_1,\dots,F_r)|\leq \prod_{k\in[r]}k^{n_k}\qquad \text{and} \qquad \sum_{k\in[r]}k\cdot n_k=\sum_{k\in[r]} |F_k|.\end{align}
\end{observation}

As the number of $r$-tuples $(F_1,\dots,F_r)$ is at most $|\cF|^r\leq 2^{rn(\log n)^{-1/18}}\ll r^{\mu(G)}\leq \kappa_{r,G}$,
a  set~$A$ which admits about as many sum-free $r$-colorings as $\kappa_{r,G}$ must give rise to a ``substantial''
$r$-tuple $(F_1,\dots,F_r)$, i.e., one for which $|\Phi(F_1,\dots,F_r)|$ is about as large as $\kappa_{r,G}$.
This information will be used to derive  stability type results  which form an important step in the proofs.
The case $r=2,3$  reads as follows. 

\begin{theorem}
\label{thm:stability23}
Suppose that $r\in\{2,3\}$,  $q\in\mathbb{N}$ and 
$0<\eps< \frac1{(q+1)}$.
Let $G$ be a type~I($q$) group of sufficiently large  order $n$ and let
 $A\subset G$ be such that $\kappa_r(A)>r^{\mu(G)-\frac{\eps n}{200}}$. Then there is a largest sum-free set $B\subset G$
such that $|A\setminus B|<\eps n$.
\end{theorem}

For $r=4,5$ and  even order groups we will need a stronger notion of stability which  is  harder to establish. One reason for the complication is that in 
these  cases the straightforward lower bound $\kappa_{r,G}\geq r^{\mu(G)}$ is  far from 
best possible, as shown by the following.

\begin{proposition}
\label{prop:lowerbounds}
Given an abelian group $G$ of even order with at least two largest sum-free sets and  let $B_1,B_2,B_3$ be largest sum-free sets in $G$ with $|\{B_1,B_2,B_3\}|\geq 2$. Then
\[\kappa_{4}(B_1\cup B_2) \geq (3\sqrt 2)^{n/2}\qquad\text{ and }\qquad \kappa_{5}(B_1\cup B_2\cup B_3) \geq 6^{n/2}.\]
\end{proposition}

Proposition~\ref{prop:lowerbounds} will be proven at the end of Section~\ref{sec:groups}. We continue with the description of the structure of ``substantial'' tuples, 
i.e., those $(F_1,\dots,F_r)$ such that $|\Phi(F_1,\dots,F_r)|$ is about as large  as the  bounds from above. The statement
of the results requires some notation.
 
Let $G$ be a finite abelian group of even order and
let $\cB=(B_1,\dots,B_t)$ be an ordered tuple of not necessarily distinct largest sum-free sets of $G$.
For a largest sum-free set $B$ let $B^1=B$ and let $B^0=G\setminus B$ be its complement. 
For an  $\veceps=(\eps_1,\dots,\eps_t)\in\{0,1\}^t$ define the \emph{atom} $\cB(\veceps)$ via
\[\cB(\veceps)=\bigcap_{i\in[t]}B_i^{\eps_i}.\]
If $\sum_{i\in[t]}\eps_i=k$ then 
$\cB(\veceps)$ is referred to as a $k$-\emph{atom}. Hence, the elements in a $k$-atom are contained in exactly $k$ of the $B_i$'s.
Finally, we say that $\cB=(B_1,\dots,B_t)$ consists of a collection of certain atoms if 
these atoms partition  $\cup_{i\in[t]} B_i$.

For four colors  the intersection structure of ``substantial'' tuples are identified as follows. 
\begin{theorem}
\label{thm:stability4}
Given an abelian group $G$ of sufficiently large even order~$n$. 
Let~$A \subseteq G$ and let $\cF=\cF(A)$ be a container family as in Theorem~\ref{thm:container}.
If $B$ is the unique largest sum-free set in $G$ and $A$ satisfies $\kappa_4(A)>3.999^{n/2}$, then $|A\setminus B|<12 n(\log n)^{-1/27}$.

If $G$ has at least two largest sum-free sets
and the tuple $(F_1,\dots,F_4)\in\cF^4$ satisfies \[|\Phi(F_1,\dots,F_4)|> \left(3\sqrt{2}-\frac1{25}\right)^{n/2},\]
then there exist three largest sum-free sets $B_1, B_2$ and $B_3=B_1\triangle B_2$ 
and a function \mbox{$f:[4]\to [3]$} such that the following holds
\begin{itemize}
\item $|F_i\setminus B_{f(i)}| <3 n (\log n)^{-1/27}$ for all $i\in[4]$, and
\item  $(B_{f(1)},\dots, B_{f(4)})$ consists of one $2$-atom, and two $3$-atoms.  
\end{itemize}
\end{theorem}

To formulate the corresponding result for five colors, we first note that
for an independent triple of largest sum-free sets $(B_1,B_2, B_3)$, there are seven largest sum-free sets contained in $B_1\cup B_2\cup B_3$.
This will be shown in Section~\ref{sec:groups} (see  Corollary~\ref{remark:reduction}).
The structure of substantial tuples for five colors then reads as follows.

\begin{theorem}
\label{thm:stability5}
Given an abelian group $G$ of sufficiently large even order~$n$. Let~$A \subseteq G$ 
and let $\cF$ be a container family of $A$ as in Theorem~\ref{thm:container}.
If $B$ is the unique largest sum-free set in $G$, 
and $A$ satisfies $\kappa_5(A)>4.999^{n/2}$, then $|A\setminus B|<15 n(\log n)^{-1/27}$.

If $G$ has at least two largest sum-free sets
and the tuple $(F_1,\dots,F_5)\in\cF^5$ satisfies \[|\Phi(F_1,\dots,F_5)|> 5.9^{n/2},\] then
there are 
three largest sum-free sets $B_1, B_2, B_3$ of $G$ such that one of the following holds.
\begin{enumerate}
 \item\label{it:stability51} 
 $B_3= B_1 \triangle B_2$ and there is a function $f: [5]\to[3]$ such that
\begin{itemize}
\item $|F_i\setminus B_{f(i)}|<3 n(\log n)^{-1/27}$ for all $i\in[5]$ and
\item $(B_{f(1)},\dots, B_{f(5)})$ consists of one $4$-atom, and two $3$-atoms. 
\end{itemize}
\item\label{it:stability52} There are four distinct largest sum-free sets $B_4,\dots, B_7$ contained in $B_1\cup B_2\cup B_3$ 
and a function $f: [5]\to [7]$ such that 
\begin{itemize}
\item $|F_i\setminus B_{f(i)}|<3 n(\log n)^{-1/27}$ for all $i\in[5]$ and
\item $(B_{f(1)},\dots, B_{f(5)})$ consists of two $2$-atoms, four $3$-atoms, and one $4$-atom. 
\end{itemize}
\end{enumerate}
\end{theorem}

 Our approach to Theorem~\ref{thm:stability4} and Theorem~\ref{thm:stability5}  is to reduce the problems we want to solve  
 for general abelian even order groups to a related problem in $\mathbb F_2^t$ for some $t\leq r\leq 5$.
This reduction and the proof of Theorem~\ref{thm:stability23}  will require further group theoretic facts which we introduce in the next section.

\section{Group theoretic facts}
\label{sec:groups}
The size and the structure of largest sum-free sets in type I abelian groups 
were determined by Diananda and Yap as follows.
\begin{theorem}[Diananda and Yap~\cite{DianandaYap}]\label{lem:mu}
Let $G$ be a type I($q$) group of order~$n$.
Then the size of the largest sum-free set $\mu(G)$ satisfies 
$\mu(G) = \left(\frac{1}{3} + \frac{1}{3q} \right) n$. 
Moreover, if $B$ is a largest sum-free set in $G$, 
then $B$ is a union of cosets of some subgroup $H$
of order $n/q$ of $G$, $B/H$ is in arithmetic progression and
$B \cup (B+B)=G$.
\end{theorem}

The following results   are due to Green and Ruzsa~\cite{GreenRuzsa}.
\begin{lemma}[Proposition~2.2~in~\cite{GreenRuzsa}] \label{lem:supersaturation}
If $G$ is abelian and $A \subset G$ contains at most $\delta n^2$  distinct Schur triples, 
then $|A| \leq \mu(G) + 2^{20} \delta^{1/5} n$.
\end{lemma}

\begin{lemma}[Lemma~4.2~in~\cite{GreenRuzsa}]\label{lem:removal}
Let $\epsilon > 0$ and let  $G$ be an abelian group. If $A \subseteq G$ has size at least $n/3 + \epsilon n$
and  contains at most $\epsilon^3 n^2 /27$ distinct sums, then there is a sum-free set $S \subseteq A$ 
such that $|S| \geq |A| - \epsilon n$.
\end{lemma}

\begin{lemma}[Lemma~5.6~in~\cite{GreenRuzsa}]\label{prop:re1}
Let~$G$ be of type I($q$).
If $S\subset G$ is a sum-free set  and $|S| > \left(\frac{1}{3} + \frac{1}{3(q+1)}\right)n$,
then $S \subset B$ for some largest sum-free set $B \subset G$.
\end{lemma}

\begin{lemma}[Lemma~5.2~in~\cite{GreenRuzsa}] \label{prop:kneser}
Suppose that $G$ is abelian, 
$S \subset G$ a sum-free set, $r > 0$ and $|S| \geq n/3 + r$. Then
there is a subgroup $H$ of $G$ with $|H| \geq 3r$, and a sum-free set $T \subset G/H$ such that 
$S \subseteq \pi^{-1}(T)$, where $\pi: G \rightarrow G/H$ is the canonical homomorphism.
\end{lemma}

Further we will need the following result from~\cite{ABMS} (see page 19).
\begin{lemma}\label{lem:proof23exact}
Let $G$ be a type I($q$) group of odd order and let $B\subset G$ be a largest sum-free set.
If $x \in G\setminus B$ , then
\[ \left| \left\{ \{a,b\} \in \binom{B}{2}: x=a+b \right\} \right| \geq \frac{n}{2q} - 1.\]
\end{lemma}

\subsection*{Even order groups}
For even order groups we need further results.  The first one refines and improves Lemma~\ref{prop:re1}  for even order groups as follows.

\begin{lemma}\label{lem:3over8}
Let~$G$ be an abelian group of even order and let
$S\subset G$ be sum-free. If $|S| > \frac{3}{8} n$ and $S$ is~not contained in a largest sum-free set of $G$,
then there exist a subgroup $H$ of $G$ with $|H|=\frac{n}{5}$ 
and a largest sum-free set $T\subset G/H$ such that
$S \subseteq \pi^{-1}(T)$, where $\pi: G \rightarrow G/H$ is the canonical homomorphism.
Moreover, for all largest sum-free sets $B$ of $G$ we have $|\pi^{-1}(T) \cap B|= \frac n5$.

In particular, if $|S| > \frac{2}{5} n$, then $S $ is contained in a largest sum-free set of~$G$.
\end{lemma}
\begin{proof}
Let $H$ be the subgroup of $G$ and $T$ be the sum-free set of $G/H$ 
obtained by the application of Lemma~\ref{prop:kneser} with $r=\frac n{24}$.
Then, $S\subset \pi^{-1}(T)$ and  $|H| = \frac{n}{\ell}$ for some $\ell\in \{2, \ldots, 8\}$.
If $\ell$ is even then $G/H$ has largest sum-free set of size $\ell/2$.
Hence, If $T$ is a largest sum-free set in $G/H$, then $S$ is contained in a largest sum-free set in $G$, a contradiction. 
If, on the other hand, $T$ is not a largest sum-free
set in $G/H$, then  trivially $|\pi^{-1}(T)|\leq \frac n\ell \left(\frac\ell2 -1\right)\leq  \frac{3}{8} n$,
a contradiction to the size of $S$.

Suppose now that $\ell \in \{3,7\}$.
Then $G/H$ is isomorphic to $\mathbb{Z}_3$ or to $\mathbb{Z}_7$.
As the largest sum-free sets in $\mathbb{Z}_3$ and $\mathbb{Z}_7$
have sizes one and two, respectively, we have $|\pi^{-1}(T)| \leq \frac{1}{3} n < \frac{3}{8} n$, which is a contradiction
to the size of $S$. Hence, $\ell=5$ which finishes the proof of the first part of the lemma.

Suppose now that $B$ is a largest sum-free set in $G$. Then  $B^0=G\setminus B$ is an index two subgroup of $G$ due to Theorem~\ref{lem:mu}. 
Since $G/ (H\cap B^0) \simeq G/H \oplus G/B^0 \simeq \mathbb{Z}_5 \oplus \mathbb{Z}_2$, 
the set $B$ must correspond to $\mathbb{Z}_5 \oplus \{1\}$ and 
$\pi^{-1}(T)$ corresponds to $\{1,4\} \oplus \mathbb{Z}_2$ or to 
$\{2,3\} \oplus \mathbb{Z}_2$.
We conclude $|\pi^{-1}(T) \cap B|= \frac n5$.
\end{proof}

\begin{lemma}
\label{lem:matching}
Let $G$ be an abelian group of even order, let
$\cB=(B_1,\dots,B_t)$ be an ordered tuple of largest sum-free sets in $G$ and
$x\in \cB(\veczero_t)$ where $\veczero_t$ is the zero vector of length $t$. Then for every $\veceps\in\{0,1\}^t$ and every 
$b\in\cB(\veceps)$, we have $x-b \in \cB(\veceps)$.
\end{lemma}
\begin{proof}
Let $b'=x-b$ and for a contradiction suppose that $b'\in\cB(\veceps')$ for an 
$\veceps'=(\eps_1',\dots,\eps_t')\neq\veceps=(\eps_1,\dots,\eps_t)$.
Then there is an index $i$ such that $\eps_i'\neq\eps_i$, i.e., 
we have that one of the two elements $b$ or $b'$ is contained in $B_i^0$
but not both. 
However, as $x\in B_i^0$, this yields a contradiction to the fact that, due to Theorem~\ref{lem:mu}, 
$B_i^0=G\setminus B_i$ is a subgroup of $G$. The lemma follows.
\end{proof}

Given an even order group $G$ and a tuple of largest sum-free sets. The following lemma
determines the size of the atoms and characterizes all largest sum-free sets contained in this tuple. 
It is a crucial part in our  reduction of the original problem in arbitrary even order groups to a related one in $\mathbb F_2^t$ for some $t\leq r$.

\begin{lemma}
\label{lem:intersectionlarg}
Let $G$ be an abelian group of even order and let $\cB=(B_1, \ldots, B_t)$, $t\geq 2$,
be an independent tuple of largest sum-free sets in $G$.
Further, let~$\veczero_t$ denote the zero vector of length~$t$.
Then the size of each atom of $\cB$ is $n/2^{t}$ and,
$B \subseteq \bigcup_{i\in [t]}B_i$ 
is a largest sum-free set in $G$ if and only if there is a largest sum-free set $S$ of $\mathbb{F}_2^t$ such that
\[ B = \bigcup_{\veceps \in S} \cB(\veceps).\]
\end{lemma}
\begin{proof}
To show the lemma we will  prove that 
the atoms $\cB(\veceps)$, $\veceps\in\{0,1\}^t$ are exactly the cosets of $\cB(\veczero_t)$, i.e., 
$G/\cB(\veczero_t)=\{\cB(\veceps)\colon \veceps\in\{0,1\}^t\}$,
that $\cB(\veczero_t)$ has size $n/2^t$
and that the map $\pi:G/\cB(\veczero_t)\to\mathbb{F}_2^t$ defined via $\pi(\cB(\veceps))=\veceps$ is a group isomorphism. 

 To see that  this implies
the conclusion of the lemma note first that a (largest) sum-free set  $S\subset \mathbb{F}_2^t$ of size $2^{t-1}$ 
maps to the sum-free set $\pi^{-1}(S)\subset G/\cB(\veczero_t)$ of the same size. 
The union $B$ of the cosets of $\pi^{-1}(S)$ is then a (largest) sum-free set in $G$ of size $n/2$. 
As $\pi^{-1}(S)$ cannot contain $\veczero_t$ we have $B \subseteq \bigcup_{i\in [t]}B_i$.    
On the other hand, if $B \subseteq \bigcup_{i\in [t]}B_i$ is a largest sum-free set in $G$, then 
$B^0=G\setminus B$ is a subgroup of $G$   due to Theorem~\ref{lem:mu}, which properly contains
$\cB(\veczero_t)$ since $t\geq 2$. By the third isomorphism theorem $B^0/\cB(\veczero_t)$ is then a non-trivial 
subgroup of $G/\cB(\veczero_t)$ and the union of the cosets
in $B'=\big(G/\cB(\veczero_t)\big)\setminus \big(B^0/\cB(\veczero_t)\big)$ is $B$. 
As $B$ is largest sum-free and each coset has size $n/2^t$ we have that $B'$ is sum-free in $G/\cB(\veczero_t)$ and has 
size $2^{t-1}$.
Hence, $B'$  corresponds to a largest sum-free set in $\mathbb{F}_2^t$ via $\pi$.

\medskip

We proceed with the proof of the lemma and first note that each non-empty atom $\cB(\veceps)$, $\veceps\in\{0,1\}^t$, is a 
subset of a coset of $\cB(\veczero_t)$ in $G$,  i.e., that 
 $a-a'\in B_i^0$ holds  for all $a,a'\in\cB(\veceps)$  and all $i\in[t]$. Indeed, recall that $B_i^0=G\setminus B_i$ is a subgroup of $G$.
 Hence, if $\eps_i=0$ then $a, a'\in B_i^0$ and therefore $a-a'\in B_i^0$. On the other hand, if $\eps_i=1$, then $a$ and $a'$ are in the largest sum-free set 
$B_i$. Thus,  $a+a'$ and $2a'$ are in its complement, the subgroup $B_i^0$. 
 Therefore $-2a'$ is also in $B_i^0$ and we have $a-a'=a+a'-2a'\in B_i^0$. This shows that $\cB(\veceps)$ is a subset of a coset of $\cB(\veczero_t)$.

Further, it is easily seen that two non-empty atoms $\cB(\veceps)\neq\cB(\veceps')$ are contained in different cosets of $\cB(\veczero_t)$.
Indeed, as $\veceps\neq\veceps'$  there is an index $i\in[t]$ such that, say, $\eps_i=0$ and
$\eps_i'=1$. Given $a\in \cB(\veceps)$ and $a'\in\cB(\veceps')$, then $a-a'$ cannot belong to $\cB(\veczero_t)$ since $B_i^0$ is a subgroup.
As the atoms  form a partition of $G$
we conclude  that they  are exactly the cosets of $\cB(\veczero_t)$.
%In other words, $G/\cB(0,\dots,0) = \{\cB(\veceps):  \veceps \in\{0,1\}^t\}$.

\medskip
To show  $|\cB(\veczero_t)|=n/2^t$ we note  that for all $k\in[t-1]$ the following holds
\[\cB(\veczero_k)=\big(\cB(\veczero_{k})\cap B_{k+1}^0\big)\cup \big(\cB(\veczero_k)\cap B_{k+1}\big)=\cB(\veczero_{k+1})\cup \cB(0,\dots,0,1).\]
Here $\cB(\veczero_{k+1})$ is a subgroup of $\cB(\veczero_k)$ and $\cB(0,\dots,0,1)$ is a non-empty set due to the lemma's assumption  that 
$B_{k+1}$ is not contained in the union of the remaining $B_i$'s. The argument from above yields therefore that $\cB(0,\dots,0,1)$ is a
coset of $\cB(\veczero_{k+1})$ in the group $\cB(\veczero_{k})$ which implies that $|\cB(\veczero_{k+1})|=|\cB(0,\dots,0,1)|=|\cB(\veczero_{k})|/2$.
With $|\cB(\veczero_1)|=|B_1^0|=n/2$, due to Theorem~\ref{lem:mu}, we obtain that $|\cB(\veczero_{t})|=n/2^t$.

It is left to verify  that the  bijective map 
from $G/\cB(\veczero_t)$ to $\mathbb{F}_2^t$ 
defined by $\cB(\veceps)\mapsto\veceps$ 
is a group homomorphism, i.e.,
 that for all $\veceps,\veceps'\in\{0,1\}^t$ we have $\cB(\veceps)+\cB(\veceps')=\cB(\veceps+\veceps')$ where the first addition is in $G/\cB(\veczero_t)$ and
the second is in $\mathbb{F}_2^t$. 
This can be verified component-wise, noting that for a fixed $i\in[t]$ we have  $B_i^0+B_i^0=B_i^0$, because $B_i^0$ is a subgroup, and that 
$B_i^1+B_i^1=B_i^0$, due to Theorem~\ref{lem:mu}. 
Further, as $B_i^1$ is the only coset of $B_i^0$ in $G$ we have $B_i^1+B_i^0=B_i^1$.
%Hence for any $a\in\cB(\veceps)$ and $a'\in\cB(\veceps')$ we have $a + a' \in  \cB(\veceps + \veceps')$. 
The lemma follows.
\end{proof}

The characterization of the largest sum-free sets in the additive group~$\mathbb{F}_2^t$
is immediate from Theorem~\ref{lem:mu} and the fact
that $\mathbb{F}_2^t$ has exactly $2^t-1$ subgroups
of index two, given by 
$\left\{ (a_1, \ldots, a_t) \in \mathbb{F}_2^t: \sum_{i \in I} a_i \equiv 0 \pmod 2\right\}$ for non-empty 
$I \subseteq [t]$.

\begin{lemma}\label{prop:chaf2t}
A subset $S \subseteq \mathbb{F}_2^t$ is a largest sum-free set in $\mathbb{F}_2^t$
if and only if there exists a non-empty $I \subseteq [t]$ such that
$$S = \left\{ (a_1, \ldots, a_t) \in \mathbb{F}_2^t: \sum_{i \in I} a_i \equiv 1\pmod2 \right\}.$$ 
\end{lemma}

Given an independent tuple $\cB=(B_1, \ldots, B_t)$, $t\geq 2$,
of largest sum-free sets in $G$. 
As a consequence of Lemmas~\ref{lem:intersectionlarg} and~\ref{prop:chaf2t} there is a correspondence between
largest sum-free sets of $B \subset \cup_{i\in[t]} B_i$ of $G$ 
and non-empty subsets $I_B \subset [t]$. The one which assigns $B_i$ to $\{i\}$ for each $i\in[t]$  is called the \emph{canonical correspondence}. It is easily seen that this indeed yields a correspondence, e.g., by induction on $t$. %In addition, we have the following corollary.
We summarize the properties of such a correspondence.

\begin{corollary}
\label{remark:reduction}
Let $G$ be an abelian group of even order and let $\cB=(B_1, \ldots, B_t)$, $t\geq 2$,
be an independent tuple of largest sum-free sets in $G$.
Consider the canonical correspondence between largest sum-free sets 
$B \subset \cup_{i\in[t]} B_i$ of $G$ 
and non-empty subsets $I_B \subset [t]$. 
Then the size of each atom of $\cB$ is $n/2^{t}$ and an atom 
$\cB(\veceps)$ is contained in $B$ if and only if 
$\veceps=(\eps_1,\dots,\eps_t)$ evaluates odd on $I_B$, i.e. $\sum_{i\in I_B}\eps_i$ is odd.
\end{corollary}
We finish the section with the proof of Proposition~\ref{prop:lowerbounds} % the lower bounds on $\kappa_{r,G}$ for $r=4,5$ and even order $G$ from . 
showing that the following holds for  sum-free sets $B_1,B_2,B_3$  in even order groups  which satisfy $|\{B_1,B_2,B_3\}|\geq 2$:
 \[\kappa_{4}(B_1\cup B_2) \geq (3\sqrt2)^{n/2}\qquad\text{ and }\qquad\kappa_{5}(B_1\cup B_2\cup B_3) \geq 6^{n/2}.\]

\begin{proof}[Proof of Proposition~\ref{prop:lowerbounds}]
Given $G$ of even order with an independent tuple $\cB=(B_1,B_2)$ of two largest sum-free sets. Consider the canonical correspondence (see Corollary~\ref{remark:reduction}), 
relating $B_i$, $i\in[2]$, to the set~$\{i\}$. 
Each atom of $\cB$ has size $n/4$ and $\cB(1,0),\cB(1,1)$ are those of $B_1$ whereas
those of $B_2$ are $\cB(0,1),\cB(1,1)$. 
 Further, there is a largest sum-free set $B_3$ corresponding to $\{1,2\}$ with the atoms $\cB(1,0)$ and $\cB(0,1)$, i.e,
$B_3=B_1\triangle B_2$.

For the first part of the proposition consider all  $\phi:B_1\cup B_2\to [4]$ 
such that $\phi^{-1}(1)\subseteq B_1$, $\phi^{-1}(2)\subseteq B_2$ and $\phi^{-1}(3),\phi^{-1}(4)\subseteq B_3$. These
colorings are clearly sum-free and there are $3^{n/2}2^{n/4}$ of them, as the elements in $\cB(1,1)$ can be colored with the colors 1 and 2, the elements in $\cB(1,0)$
with colors 1, 3 and 4 and   the elements in $\cB(0,1)$ with colors 2, 3 and 4.

For the second part first consider the case that $B_1$ and $B_2$ are two sum-free sets and $B_3=B_1\triangle B_2$ as above. 
Consider the colorings $\phi:B_1\cup B_2\to [5]$ such that $\phi^{-1}(1),\phi^{-1}(2)\subseteq~B_1$, $\phi^{-1}(3),\phi^{-1}(4)\subseteq B_2$
and $\phi^{-1}(5)\subseteq B_3$. This gives rise to $3^{n/2}4^{n/4}$  sum-free 5-colorings.

Finally, given the independent triple $\cB=(B_1,B_2,B_3)$ with atoms of size $n/8$ each,
because of Corollary~\ref{remark:reduction}. Consider the canonical correspondence %described in Remark~\ref{remark:reduction}, 
relating $B_i$, $i\in[3]$, to the set $\{i\}$. Then there are four further largest sum-free sets contained in $B_1\cup B_2\cup B_3$. One of these is $B_4=B_1\triangle B_2$
corresponding to $\{1,2\}$ and consisting of the atoms $\cB(1,0,0), \cB(1,0,1),\cB(0,1,0),\cB(0,1,1)$. 
An another one is $B_5=(B_1\triangle B_2\triangle B_3)\cup (B_1\cap B_2\cap B_3)$ corresponding to $\{1,2,3\}$
with the atoms $\cB(1,0,0), \cB(0,1,0),\cB(0,0,1),\cB(1,1,1)$.
Consider  all  colorings $\phi: B_1\cup B_2\cup B_3\to[5]$ such that $\phi^{-1}(i)\subset B_i$, $i\in[5]$. This gives rise to 
$3^{4\cdot \frac n8}2^{2\cdot\frac n8}4^{\frac n8}$ sum-free 5-colorings. 
\end{proof}

\section{Proofs of  stability: Theorem~\ref{thm:stability23}, \ref{thm:stability4}, \ref{thm:stability5}}
With the group theoretic results from previous section we are now ready to give the proofs
of  Theorem~\ref{thm:stability23}, Theorem~\ref{thm:stability4} and Theorem~\ref{thm:stability5}. 

\begin{proof}[Proof of Theorem~\ref{thm:stability23}]
We give the proof for $r=3$ only. The proof for $r=2$ follows the same line.
For given $\eps>0$ let $\delta=\eps/200$ and let $|G|=n$ be sufficiently large.

Let $A\subset G$ with $\kappa_3(A)\geq 3^{\mu(G)-\delta n}$  and let  $\cF=\cF_A$ be a container family as in Theorem~\ref{thm:container}. 
According to \eqref{it:container3} of this theorem   each $F_i$ contains at most $n^2(\log n)^{-1/9}$ Schur triples  
and  Lemma~\ref{lem:supersaturation} implies for large enough $n$ that $|F_i|\leq \mu(G)+3n(\log n)^{-1/27}$ for each $F_i\in\cF$.
Let $(F_1,F_2,F_3)$ be a triple  maximizing $|\Phi(F'_1,F'_2,F'_3)|$ over all $(F'_1,F'_2,F'_3)\in\cF^3$. 
Together with~\eqref{eq:PhiInContainer1} of Observation~\ref{obs:PhiInContainer}  we infer $3^{\mu(G)-\delta n}\leq \kappa_3(A)\leq |\cF|^3\cdot |\Phi(F_1,F_2,F_3)|$.
Further,~\eqref{it:container1} of Theorem~\ref{thm:container}  states that $\log_2 |\cF|\leq n(\log n)^{-1/18}$ which implies $\log_3|\Phi(F_1,F_2,F_3)|\geq \mu(G)-2\delta n$
 for large enough $n$.

Let $n_i$, $i=1,2,3$, be the number of elements contained in exactly~$i$ members from $(F_1,F_2,F_3)$.
Then $n_1+2n_2+3n_3= |F_1|+|F_2|+|F_3|\leq 3\mu(G)+3\delta n$, hence, $n_2\leq \frac{3}{2}(\mu(G)+\delta n-n_3)$.
We conclude that $n_3 \geq \mu(G)- 59 \delta n$ must hold otherwise the first part of~\eqref{eq:PhiInContainer2} of Observation~\ref{obs:PhiInContainer} 
together with $\frac 32\log_32<\frac{19}{20}<1$ would yield the following contradiction
\begin{align*}
\mu(G)-2\delta n\leq \log_3 |\Phi(F_1,F_2,F_3)| \leq \,\, &  n_3 +\big(\mu(G)+\delta n - n_3\big)\frac{3}{2}\log_32\\
< \,\, &  \mu(G) - 59\delta n + 60 \delta n\cdot \frac{3}{2}\log_32< \mu(G) - 2\delta n.
\end{align*}
Let $F=F_1\cap F_2\cap F_3$ and note that
since $A=F_1 \cup F_2 \cup F_3$ and $|F_i\setminus F|= |F_i|-|F|\leq 60 \delta n$ for $i=1,2,3$, we have $|A\setminus F|\leq 180\delta n$. 
To conclude the proof all is needed is to  show that $|F\setminus B|\leq \delta n$ for some largest sum-free set $B$ since this then implies $|A\setminus B|\leq181\delta n<\eps n$.
Note to this end that $|F|=n_3\geq \mu(G)-59\delta n=\left(\frac 13+\frac1{3q}-59\delta\right)n>\frac n3+\delta n$ due to  $\delta\leq  \frac{1}{200(q+1)}$. Lemma~\ref{lem:removal} then
implies that for sufficiently large $n$ there is a  sum-free set $S$ of size 
$|S|\geq |F|-\delta n\geq \mu(G)-60\delta n > \left(\frac 13+\frac1{3(q+1)}\right)n$ that is contained in $F$.
By Lemma~\ref{prop:re1} there must exist a largest  sum-free set $B\subset G$ such that $S\subset B$ showing that $|F\setminus B|\leq\delta n$, as claimed. 
\end{proof}

The proofs of Theorem~\ref{thm:stability4} and Theorem~\ref{thm:stability5} will be  more involved and we divide the argument into two parts reflected by  
two lemmas, Lemma~\ref{lem:alllargest} and Lemma~\ref{lem:OptStructure}.
In the following we state the two lemmas which will immediately imply
Theorem~\ref{thm:stability4} and Theorem~\ref{thm:stability5}. The proofs of the lemmas will be shown subsequently.

The first lemma states that $|\Phi(F_1,\dots,F_r)|$ being large implies that each of the $F_i$'s is essentially contained in a largest sum-free set in $G$. 

\begin{lemma}
\label{lem:alllargest}
Let $G$ be an abelian group of sufficiently large even order $n$. Given  
 $A\subset G$ and  $\cF=\cF(A)$ as in Theorem~\ref{thm:container} and
suppose that $G$, $r$ and $(F_1,\dots,F_r)\in\cF^r$ are  such that
%\footnote{Note that in the case that $G$ has at least two largest sum-free sets the lower bounds 
%yields  $\kappa_{4,G}\geq 3^{n/2}2^{n/4}>2.0597^n$ and $\kappa_{5,G}\geq 6^{n/2}>2.449^n$.}
\begin{itemize}
\item $G$ has a unique largest sum-free set and  $|\Phi(F_1,\dots, F_r)|>\left(\sqrt{r}-\frac1{1000}\right)^{n}$ or
\item $r=4$ and $G$ has at least two largest sum-free sets and $|\Phi(F_1,\dots, F_r)|>2.01^n$ or
\item $r=5$ and $G$ has at least two largest sum-free sets and $|\Phi(F_1,\dots,F_r)|>2.41^n$.
\end{itemize}

Then 
there is an $r$-tuple of largest sum-free sets $(B_1,\dots,B_r)$ in $G$ such that
\[|F_i\setminus B_i|\leq3 n(\log n)^{-1/27}\quad\text{ for all $i\in[r]$}.\]
\end{lemma}

Note that Theorem~\ref{thm:stability4} and Theorem~\ref{thm:stability5} immediately follows
in case $G$ has a unique largest sum-free set.
For the case that $G$ has at least 
two largest sum-free sets we put forth the following lemma, which determines the intersection  structure 
of the optimal configuration of the $B_i$'s.
\begin{lemma}\label{lem:OptStructure}
Given an abelian group $G$ of even order.
% and define $\eta_4=\log_32$, $\eta_5=\frac{13}{16}$,  $\theta_4=\frac14$ and $\theta_5=\frac12$.

If  $\mathcal{B}  = (B_1, \ldots, B_4)$ is tuple of largest sum-free sets 
such that  
 $|\Phi({B_1,\dots,B_4})| > 2^{n}$,
 then there are two sets among $B_1,\dots,B_4$ whose union contains all $B_i$, $i=1,\dots,4$, and
 $\mathcal{B}$ consists of one 2-atom and two 3-atoms. In particular, $|\Phi(\mathcal{B})|= (3\sqrt2)^{n/2}$.
\medskip

 If  $\cB=(B_1,\dots,B_5)$ is a tuple of largest sum-free sets  with 
 $|\Phi({B_1,\dots,B_5})|>\left(3\cdot2^{7/2}\right)^{n/4}.$
Then either
\begin{itemize}
 \item there are two sets among $B_1,\dots,B_5$ whose union contains all $B_i$, $i=1,\dots,5$, and 
  $\mathcal{B}$ consists of two 3-atoms and one 4-atom, or
 \item  there are three sets among $B_1,\dots,B_5$ whose union contains all $B_i$, $i=1,\dots,5$, and
  $\mathcal{B}$ consists of two 2-atoms, four 3-atoms and one 4-atom.
 \end{itemize}
 In both cases we have $|\Phi(\cB)|=6^{n/2}$.
\end{lemma}

Theorem~\ref{thm:stability4} and Theorem~\ref{thm:stability5} are now easy consequences of the two lemmas.

\begin{proof}[Proof of Theorem~\ref{thm:stability4} and Theorem~\ref{thm:stability5}]
Given an abelian group $G$ of sufficiently large even order~$n$ and let  $\gamma=   3{(\log n)^{-1/27}}$.
Let $A\subset G$, $\kappa_r(A)$ and $\cF=\cF(A)$ be as stated in Theorem~\ref{thm:stability4} or Theorem~\ref{thm:stability5}, respectively.
In the case that $G$ has a unique largest sum-free set $B$  let $(F_1,\dots,F_r)$ be the tuple among all $(F_1',\dots,F_r')\in\cF^r$  which maximizes $|\Phi(F_1',\dots, F_r')|$.
Owing to \eqref{eq:PhiInContainer1} of Observation~\ref{obs:PhiInContainer} we have 
$\kappa_r(A)=|\Phi_r(A)|\leq |\cF|^r|\Phi(F_1,\dots,F_r)|$. As $|\cF|^r\leq 2^{rn(\log n)^{-1/18}}\ll \kappa_r(A)$
we conclude for sufficiently large $n$ that  $|\Phi(F_1,\dots, F_r)|> |\cF|^{-r}\left(r-\frac1{1000}\right)^{n/2}>\left(\sqrt{r}-\frac 1{1000}\right)^n$ for 
$r=4,5$. 

 On the other hand, if $G$ has at least two largest sum-free sets let $(F_1,\dots,F_r)$  be a tuple given by the theorems, i.e.,
\begin{itemize}
\item for $r=4$ we have $|\Phi(F_1,\dots,F_4)|> \left(3\sqrt 2-\frac1{25}\right)^{n/2}> 2.01^n > 4^{4\gamma n} \cdot 2^{n}$  and
\item  for $r=5$ we have $|\Phi(F_1,\dots,F_5)|>5.9^{n/2}>5^{5\gamma n}\left(3\cdot2^{7/2}\right)^{n/4}>2.41^n$. 
\end{itemize}

In particular,  the presumptions of Lemma~\ref{lem:alllargest} are satisfied  in all cases and we infer
that there is a tuple $\cB=(B_1,\dots, B_r)$ of  largest sum-free sets in $G$ such that 
$|F_i\setminus B_i|\leq \gamma n$ for all $i\in[r]$. 
As $A=F_1\cup \dots\cup F_r$ we have $|A\setminus B|\leq r\gamma n$ and the theorems follow in the case $G$ has a unique largest sum-free set.

Suppose now that $G$ has at least two largest sum-free sets. We aim to apply Lemma~\ref{lem:OptStructure} to~$\cB$ from above and therefore now verify its presumption.
Define for all $i\in[r]$ the sets $F_i^1=F_i\cap B_i$ and $F_i^0=F_i\setminus B_i$.
Then 
$|\Phi(F_1,\dots, F_r)|\leq|\Phi(F_1^0,\dots, F_r^0)|\cdot |\Phi(F_1^1,\dots, F_r^1)|$ and 
$|\Phi(F_1^0,\dots, F_r^0)|\leq r^{r\gamma n}$. Further,  any coloring 
$\phi\in\Phi(F_1^1,\dots, F_r^1)$ can be extended to a coloring $\phi'\in \Phi(B_1,\dots, B_r)$ 
as $F_i^1\subset B_i$ and $B_i$ is sum-free. This implies $|\Phi(F_1^1,\dots, F_r^1)|\leq  |\Phi(B_1,\dots, B_r)|$ and we 
infer $|\Phi(\cB)|\geq r^{-r\gamma n}|\Phi(F_1,\dots, F_r)|$. Considering the bounds on $|\Phi(F_1,\dots, F_r)|$ from above, this shows
that $\cB$ satisfies the presumptions of Lemma~\ref{lem:OptStructure} and
 the theorems follow from the conclusions of this 
lemma for the case $G$ has at least two largest sum-free sets.
\end{proof}

\subsection{Proofs of Lemma~\ref{lem:alllargest} and Lemma~\ref{lem:OptStructure}}
We first show the proof of Lemma~\ref{lem:OptStructure}. Note that  
Corollary~\ref{remark:reduction} reduces the problem to
a finite task which one could check exhaustively. For completeness we give a self-contained proof.

\begin{proof}[Proof of Lemma~\ref{lem:OptStructure}]

Given the tuple $\cB=(B_1,\dots,B_r)$ with the properties stated in the lemma. 
Let  $n_i = \sum_{\veceps \,: \, \sum \eps_j =i } |\mathcal{B}(\veceps)|$ be the number of elements contained in $i$-atoms, i.e., those belonging to exactly
$i$ sets from $\cB$. Recall from \eqref{eq:PhiInContainer2} of Observation~\ref{obs:PhiInContainer} that 
\begin{align}\label{eq:phibis0}\log_3 |\Phi(\mathcal{B})|\leq  \sum_{i \in [r]} {n_i}\log_3 i \qquad \text{ and }\qquad  \sum_{i\in[r]}i \cdot n_i=\sum_{i\in[r]}|B_i|=r\mu(G)=\frac {rn}2.
\end{align}
As $\frac i3\geq\log_3i$ for all integer $i\geq 1$ we derive  
for $\eta_4=\log_32$ and $\eta_5=13/16$ that
\begin{equation} \label{eq:phibis}
\eta_rn< \log_3 |\Phi(\mathcal{B})|\leq  \sum_{i \in [r]} {n_i}\log_3 i\leq \frac13\sum_{i\in[r]}i\cdot n_i - \frac{1}{3} n_1 \leq \frac13\left(\frac{r n}2-n_1\right). 
\end{equation}
Let $t\in[r]$ be the largest number such that there is an independent $t$-tuple of elements in $\cB$, which we denote by $\cA=(A_1,\dots,A_t)$.
Recall the correspondence of the structure of $\cA$ to that of $\mathbb F_2^t$ as described in Corollary~\ref{remark:reduction}. 
According to this each largest sum-free set $B$ contained in $\cup A_i$ corresponds to a largest sum-free set in $\mathbb F_2^t$ and is associated to a set $I=I(B)\subset [t]$.
The atoms of $\cA$ then correspond  to elements $\veceps\in \mathbb F_2^t$ and an atom $\cA(\veceps)$ is contained in $B$  if and only if 
 $\veceps=(\eps_1,\dots,\eps_t)$ evaluates odd on $I=I(B)$, i.e. $\sum_{i\in I}\eps_i$ is odd. 
Throughout the proof we will consider the canonical correspondence assigning $A_i$, $i\in[t]$, to the set $I=\{i\}$, so that, the 1-atoms correspond to those $t$ elements 
$\veceps\in\mathbb F_2^t$
with exactly one $1$-entry.

Using this correspondence   we first show that $1<t\leq r-2$ and  $n_1=0$. The bound $t>1$ is trivial, since otherwise 
we would have $|\Phi(\cB)|=r^{\frac n2}\leq3^{\eta_rn}$ which contradicts \eqref{eq:phibis}.
Suppose that $t=r$. 
 Then the atoms have size $\frac n{2^r}$ due to Corollary~\ref{remark:reduction} and the number of 1-atoms is $r$. Hence,
$n_1=r\frac n{2^r}$ which yields the contradiction $\log_3|\Phi(\cB)|\leq\eta_r n$. On the other hand, if $t\leq r-1$ and $n_1>0$ then 
$n_1$ has the size of at least one atom, i.e., $n_1\geq n/2^{r-1}$ 
which again yield a violation of $\log_3|\Phi(\cB)|>\eta_r n$.  Having $t=r-1$, say $A_1=B_1,\dots, A_{r-1}=B_{r-1}$,  and $n_1=0$, 
however, means that the remaining  largest sum-free set $B\in \{B_1,\dots, B_r\}\setminus \{A_1,\dots,  A_{r-1}\}$ must cover all 1-atoms of $\cA$.
Since $A_i$ corresponds to $\{i\}$ the set $B$  must correspond to $I=[t]$. 
This determines the atom structure of $\cB$ completely and it is readily checked that 
 $\cB$ then consists of six $2$-atoms and one $4$-atom in case $r=4$ and 
of ten $2$-atoms and five $4$-atoms in case $r=5$. With \eqref{eq:phibis0} this yields a contradiction to the lower bound $\log_3|\Phi(\cB)|>{\eta_r n}$. 

Hence, let $1<t\leq r-2$ and $n_1=0$ in both cases $r=4$ and $r=5$. 
For $t=2$ and $r=4$ let $a_i$, $i=2,3,4$, denote the number of $i$-atoms of $\mathcal{B}=(B_1,\dots,B_r)$.
Similarly define $b_i$, $i=2,\dots,5$ for $t=2$, $r=5$ and $c_i$, $i=2,\dots,5$, for $t=3$ and $r=5$. 
Due to Corollary~\ref{remark:reduction} we know that in the case $t=2$ ($t=3$, respectively) the set $\cup_{i \in [r]}B_i$ is the union of three (seven, respectively)
atoms, each of size $n/2^t$. This and  the second part of \eqref{eq:phibis0} implies that 
\begin{align}\label{eq:sys}
 \sum_{i=2}^4a_i=\sum_{i=2}^5b_i=3,\quad \sum_{i =2}^{5}c_i =7,\qquad \sum_{i=2}^4i a_i=8,\qquad  \sum_{i=2}^{4}i b_i=10, \qquad\sum_{i=2}^{5}i c_i =20. 
\end{align}
Solving the equations yields that $(a_2,a_3,a_4)\in\{(2,0,1),(1,2,0)\}$.
If $(a_2,a_3,a_4)=(2,0,1)$ then the first part of~\eqref{eq:phibis0} yields 
$|\Phi(\cB)|\leq (4\cdot4)^{n/4}=2^{n}$ which violates the lower bound
on $|\Phi(\cB)|$. For $(a_2,a_3,a_4)=(1,2,0)$ we have $|\Phi(\cB)|\leq (2\cdot 9)^{n/4}$.
Further,  by considering largest sum-free sets in 
 $\mathbb F_2^2$ corresponding to $I_1=\{1\}$, $I_2=\{2\}$, $I_3=I_4=\{1,2\}$,  
 one obtains via the canonical correspondence 
a tuple $(B_1,\dots,B_4)$ which consists of one $2$-atoms and two $3$-atoms, hence realizing  the optimum  $(a_2,a_3,a_4)=(1,2,0)$.
The first part of the lemma follows.

Turning to the second part of the lemma we first solve \eqref{eq:sys} for  $(b_2,\dots,b_5)$ and obtain  its solution set $\{(1,1,0,1),(1,0,2,0),(0,2,1,0)\}$.
It is readily checked that $|\Phi(\cB)|\leq 6^{n/2}$ where the maximum is achieved only for $(b_2,\dots,b_4)=(0,2,1,0)$. The second largest value is attained by
$(1,0,2,0)$ which would yield $|\Phi(\cB)|\leq (2\cdot 16)^{n/8}$, a contradiction to the lower bound. 
Further, by considering  largest sum-free sets in $\mathbb F_2^2$ corresponding to $I_1=I_2=\{1\}$, $I_3=I_4=\{2\}$
and $I_5=\{1,2\}$ we see that the optimum $(b_2,\dots,b_4)=(0,2,1,0)$ can indeed be realized.

Finally, we solve \eqref{eq:sys} for $(c_2,\ldots, c_5)$ and obtain its solution set $\{(1, 6, 0, 0), (2, 4, 1, 0),(3, 2, 2, 0),$ 
$(3, 3, 0, 1),(4, 0, 3, 0), (4, 1, 1, 1), (5, 0, 0, 2) \}$. However, the maximum is attained by $(1, 6, 0, 0)$,
which cannot be realized as 
atom structure. To see this, suppose that $\mathcal{B}=(B_1,\cdots,B_5)$ has this atom structure. Assume that  
$\cB'=(B_1, B_2, B_3)$ is independent and corresponds to $I_i=\{i\}$. 
Note that $\cB'(1,1,1)$ is a 3-atom, hence $B_4$ and $B_5$ 
cannot contain $\mathcal{B'}(1,1,1)$ since $\cB(1,1,1)$ would be a 4-atom or a 5-atom otherwise.
However, largest sum-free sets of $\mathbb{F}_2^3$ which do not contain $(1,1,1)$
correspond to two elements sets $I\subset[3]$. In particular, they consist of two 1-atoms  and two 2-atoms of $\mathcal{B}'$. As there are only three 2-atoms in $\cB'$  
we conclude that $\mathcal{B}$ would contain a 4-atom.

Within the solutions for $(c_2,\dots,c_5)$, the third  largest
value is achieved by $(3,2,2,0)$ giving $|\Phi(\cB)|\leq (8\cdot9\cdot16)^{n/8}$, a contradiction to the lower bound.
The second largest is obtained by
$(2,4,1,0)$ giving $|\Phi(\cB)|\leq (4\cdot 81\cdot 4)^{n/8}=6^{n/2}$.
Moreover, the five largest sum-free sets corresponding to $I_i=\{i\}$, $i\in[3]$, $I_4=\{1,2\}$ and $I_5=\{1,3\}$ indeed consists of
two 2-atoms, four 3-atoms and one 4-atom showing that $(2,4,1,0)$ is indeed realizable. The lemma follows.
\end{proof}

\begin{proof}[Proof of Lemma~\ref{lem:alllargest}] 
Let $\gamma=3(\log n)^{-1/27}$.
Given an even order abelian group $G$ of sufficiently large order $n$ together with a container family $\cF$ as in Theorem~\ref{thm:container} and let 
$(F_1,\dots,F_r)$ be a  tuple as in the lemma. 
A set $F$ is called good if there 
is a largest sum-free set $B$ in $G$ such that $|F\setminus B| \leq \gamma n$. 
Hence, to establish the lemma we need to show that all $F_i$, $i=1,\dots,r$, are good.

Note to this end that $|F_i|> \frac 25n+\gamma n$ readily implies that $F_i$ is good.
Indeed, by property \eqref{it:container3} of Theorem~\ref{thm:container} we know that $F_i$
contains at most $n^2(\log n)^{-1/9}$ Schur triples. Lemma~\ref{lem:removal}  then implies that 
one can obtain a sum-free set $S_i\subset F_i$ by removing at most $\gamma n$ elements from $F_i$. 
As $|S_i|>\frac 25n$ we conclude by the last part of Lemma~\ref{lem:3over8} that there is a largest sum-free set $B_i$ such that 
$S_i\subset B_i$, showing that $F_i$ is  good.

For any given tuple $(F_1,\dots, F_r)\in\cF^r$ of containers let $n_i$ denote the number of elements contained in exactly $i$ sets 
from $\{F_1,\dots,F_r\}$. Recall from~\eqref{eq:PhiInContainer2} that
\begin{align}\label{eq:trivcol}
|\Phi(F_1,\dots,F_r)|\leq \prod_{i\in[r]}i^{n_i}\qquad \text{and}\qquad \sum_{i\in[r]}i\cdot n_i=\sum_{i\in[r]}|F_i|
\end{align}
holds. Using $\frac i3\geq\log_3i$ for all integer $i\geq 1$ we derive
\begin{align}\label{eq:n3triv}\log_3|\Phi(F_1,\dots,F_r)|\leq \frac13\sum_{i\in[r]}|F_i|.\end{align}
Further, by~\eqref{it:container3} of Theorem~\ref{thm:container} and Lemma~\ref{lem:supersaturation}  we have $|F_i|\leq\frac n2+2^{20}n(\log n)^{-1/45}$ for all $i\in[r]$.

We first consider the case that $G$ has at least two largest sum-free sets for which the lemma easily follows.
Indeed, suppose that  there is an $i\in[r]$ such that $|F_i|\leq\frac 25n+\gamma n$. Then we have 
\[\sum_{i\in[r]}|F_i|\leq (r-1)\left(\frac n2+2^{20}\frac n{(\log n)^{\frac1{45}}}\right)+\frac 25n+\gamma n,\]
which, with \eqref{eq:n3triv} and sufficiently large $n$, yields $|\Phi(F_1,\dots,F_r)|< 2.01^{n}$ in the case $r=4$ and 
$|\Phi(F_1,\dots,F_r)|< 2.41^{n}$ in the case $r=5$. This, however, yields a contradiction to the fact that $(F_1,\dots,F_r)$ is substantial and we conclude 
that $|F_i|>\frac 25n+\gamma n$ for all $i\in[r]$ implying that $F_i$ is good, as claimed.\footnote{We note that the same argument also works for $r=6$ in conjunction with
the lower bound $\kappa_{r,G}\geq \kappa_6(B_1\cup B_2\cup B_3)\geq 2^{3n/4} 3^{n/2} \geq 2.9129^n$ obtained by an independent triple $(B_1,B_2,B_3)$ of largest sum-free sets. 
Assuming that a there is an $i\in[6]$ with $|F_i|\leq \frac 25n+\gamma n$ would imply
$|\Phi(F_1,\dots,F_r)|< 2.91^{n}$.}

\medskip

Next we prove the lemma in case the group has a unique largest sum-free set which is somewhat more complicated as the lower bound is significantly smaller in this case. 
Let $G$ be a group of even order with the unique largest sum-free set $B$ and 
let $(F_1,\dots,F_r)$ be a substantial tuple, i.e., $|\Phi(F_1,\dots,F_r)|\geq (\sqrt{r}-0.01)^n$.  
Note first that if $F_1,\dots,F_{r-1}$ of the tuple $(F_1,\dots,F_r)$ are good, then with $F_i^1=F_i\cap B$ we have
\begin{align}
\label{eq:allgood} (r^{\frac12}-0.01)^{n}\leq |\Phi(F_1,\dots, F_r)|   \leq & \,\,r^{r\gamma n} \,\,|\Phi(F_1^1,\dots, F_{r-1}^1, F_r)| \leq \,\,  r^{r \gamma n} r^{|F_r\cap B|}(r-1)^{|B\setminus F_r|}.
\end{align}
This implies that $F_r$ must have size larger than $\frac 25n+\gamma n$ for sufficiently large $n$ and hence being good as well.

Further, by the same argument as above, i.e., using  $|F_i|\leq\frac n2+2^{20}n(\log n)^{-1/45}$ for all $i\in[r]$ together with property ~\eqref{eq:n3triv}  we derive the following properties. 
\begin{claim}\label{claim:uniqueB}\mbox{}
\begin{enumerate}
\item \label{it:unique1}Let $k_4=3$ and $k_5=2$ then at least $k_r$ sets from $(F_1,\dots,F_r)$ have size larger than $\frac 25n+\gamma n$, implying that they are good,
\item \label{it:unique2} at least four sets from $(F_1,\dots,F_5)$ have size larger than $\frac n3+\gamma n$, and
\item \label{it:unique3}if  three sets from $(F_1,\dots,F_5)$ have size at most $\frac 25n+\gamma n$ then all five sets have size larger than $\frac38n+\gamma n.$ 
\end{enumerate}
\end{claim}
In particular, as $k_4=3$ the lemma follows in the case $r=4$ and it is  left to prove that for 
$r=5$ there are four sets out of $(F_1,\dots,F_5)$ which have size larger than $\frac 25n+\gamma n$. 

 For a given tuple $(F_1,\dots, F_5)$ let $F_i^1=F_i\cap B$ and $F_i^0=F_i\setminus B$  and let $m_i$ be the number of elements contained in exactly $i$ sets from $(F_1^1,\dots,F_5^1)$. 
Then
\begin{align}
\label{eq:2good}
|\Phi(F_1,\dots, F_5)|\leq |\Phi(F_1^0,\dots, F_5^0)|\cdot |\Phi(F_1^1,\dots, F_5^1)| \qquad\text{and}\qquad m_1+\dots+m_5=|B|=\frac n2.
\end{align}
Note also that $ |\Phi(F_1^1,\dots, F_5^1)|\leq  |\Phi(B,\dots, B,F_i^1,\dots, F_5^1)|$, $i\in[5]$, 
since $F_j^1\subset B$, for all $j\in[5]$.

By~\eqref{it:unique1} of Claim~\ref{claim:uniqueB} we  may assume  that $F_1$ and $F_2$ are good, and for a contradiction suppose that
$|F_i|\leq \frac 25n+\gamma n$, $i=3,4,5$.
Then due to~\eqref{it:unique3} of Claim~\ref{claim:uniqueB}  
each of these $F_i$'s  has size larger than $\frac 38n + \gamma n$.
Property~\eqref{it:container3} of Theorem~\ref{thm:container}
and Lemma~\ref{lem:removal}  implies that for each such $i$ there is a sum-free set $S_i \subseteq F_i$
such that $|F_i \setminus S_i| \leq \gamma n$.
By Lemma~\ref{lem:3over8} there is a sum-free set $L_i$ containing $S_i$  such that  $|L_i\cap B|=|L_i\setminus B|=\frac n5$. Therefore
 $|F_i^0|, |F_i^1|\leq \frac n5+\gamma n$ for $i \in \{3,4,5\}$ and~(\ref{eq:n3triv}) yields 
$|\Phi(F_1^0,\dots,F_r^0)|\leq 3^{\frac{n}{5}+2\gamma n }.$

Consider now the tuple $(B,B,F_3^1,F_4^1 ,F_5^1)$ with  $f=|F_3^1|+|F_4^1|+|F_5^1|\leq \frac 35n+ 3\gamma n$ and let  $m_i$ be defined as above. Then   
$m_3+2m_4+3m_5=f$ and together with the second part of~\eqref{eq:2good} this implies  $m_4 + 2 m_5-m_2+\frac n2-f= 0$. 
Using the first part of \eqref{eq:2good} and $\sqrt{\frac{20}{27}}<\frac {8}{9}$ 
we obtain for $G$ of sufficiently large order
\begin{align*}
|\Phi(F_1,\dots, F_5)| \leq &\,\,
|\Phi(F_1^0, \dots, F^0_5)| \,\, |\Phi(B, B, F^1_3,F^1_4, F^1_5)|  \\
\leq & \,\, 3^{\frac n5+2\gamma n}~2^{m_2}~3^{m_3}~4^{m_4}~5^{m_5}\\
= &\,\, 3^{\frac n5+2\gamma n}~2^{m_4 +2 m_5 + \frac n{2}-f}~3^{f - 2m_4 -3m_5}~4^{m_4}~5^{m_5}\\
\leq &~~3^{ \frac n5+2\gamma n}~2^{\frac n{2}-f} \,\, 3^{f} \left(\frac 89\right)^{m_4}\left(\frac{20}{27}\right)^{m_5}\\
\leq& \,\,3^{ \frac n5+2\gamma n}~2^{\frac n{2}-f} \,\, 3^{f}  \, \left(\frac{8}{9}\right)^{f-\frac n2}\\
\leq &\,\, 3^{ \frac n5+2\gamma n} \left(\frac{4}{3}\right)^{f}  \left(\frac{3}{2}\right)^{n}
\,\, < \,\, 2.23^{n}.
\end{align*}
This contradicts the lower bound on $|\Phi(F_1,\dots, F_5)|$ and we conclude that one set among $F_3, F_4, F_5$, say $F_3$, has size larger than 
$\frac 25n+\gamma n$  and is therefore good.

Finally we show that $F_4$ or $F_5$ is good, assuming that $F_1, F_2, F_3$ are. Let
 $f_0=|F_4^0\cap F_5^0|$ and note that  $|\Phi(F_1^0,\dots,F_5^0)|\leq 5^{4\gamma n}2^{f_0}$ due to the goodness of $F_1$, $F_2$ and $F_3$. 
Further, consider $m_1,\dots,m_5$ for the tuple  $(B,B,B,F_4^1,F_5^1)$. 
Then with $f=|F_4|+|F_5|$ we have $m_1=m_2=0$ and
\begin{align}
\label{eq:F4m}
 m_3+m_4+m_5=\frac n2, \quad m_4+2m_5=|F_4^1|+|F_5^1|,\quad\text{and}\quad |F_4^1|+|F_5^1|+2f_0\leq f.
\end{align}
Using  $\sqrt{5/3}< \frac 43$ we obtain 
\begin{align}
\begin{split}
\label{eq:F4}
|\Phi(F_1,\dots, F_5)|   
\leq &\,\, 
|\Phi(F_1^0, \dots, F_5^0)| \,\, |\Phi(B,B,B,F^1_4, F_5^1)|  \\
\leq &\,\, 5^{4\gamma n} 2^{f_0} \,\, 3^{\frac n2 - m_4 - m_5} 
\,\, 4^{m_4} \,\,5^{m_5}\\
\leq&\,\, 5^{4 \gamma n} 2^{f_0} \,\, 3^{\frac n2}  \,\, \left(\frac 43\right)^{f -2f_0}\\
\leq&~~5^{4 \gamma n} \left(\frac 98\right)^{f_0} \,\, 3^{\frac n2}  \,\, \left(\frac 43\right)^{f}.
\end{split}
\end{align}	
This easily implies that $F_4$ or $F_5$ must have size larger than $\frac38n+\gamma n$.
Indeed, otherwise $f\leq\frac 34n+2\gamma n$ and~\eqref{it:unique2} of Claim~\ref{claim:uniqueB} implies that either $F_4$ or $F_5$ has size larger than $\frac n3 + \gamma n$.
By~\eqref{it:container2} of Theorem~\ref{thm:container} and
Lemma~\ref{lem:removal} one can then make, say, $F_4$ sum-free by removing at most $\gamma n$ elements. 
This shows that $f_0\leq|F_4^0|\leq \frac n4+\gamma n$ as the largest sum-free set in $B^0$ has size at most $|B_0|/2=n/4$.
Plugging these bounds for $f$ and $f_0$ into \eqref{eq:F4}, however, yields 
$|\Phi(F_1,\dots, F_5)| <2.22^n$ for $G$ with sufficiently large order.  This contradicts the lower bound, hence, we can assume that $F_4$ has size larger than $\frac38n + \gamma n$. 

Finally assume that $|F_4|, |F_5|\leq \frac 25 n+\gamma n$. Then $f\leq \frac 45+2\gamma n$ and  by using~\eqref{it:container3} 
of Theorem~\ref{thm:container} with Lemma~\ref{lem:removal}  and Lemma~\ref{lem:3over8} 
we further conclude that there exist  sum-free sets $S_4$ and $L_4$ with $S_4\subset L_4$ such that 
$|F_4 \setminus S_4| \leq  \gamma n$ and $|L_4\cap B|=|L_4\setminus B|=\frac n5$. This implies $f_0\leq |F_4^0|\leq |L_4\cap B^0|+\gamma n\leq\frac n5+\gamma n$ 
and we obtain from \eqref{eq:F4} for $G$ of sufficiently large order that $|\Phi(F_1,\dots, F_5)|<2.233^n$. This again contradicts the lower bound and
we conclude that $F_4$ is good and so is $F_5$ as shown in the beginning. This finishes the proof of the lemma.
\end{proof}

\section{Exact results}
\label{sec:exact}
In this section we prove  Theorem~\ref{thm:main23} and Theorem~\ref{thm:main45} based on Theorem~\ref{thm:stability23},~\ref{thm:stability4} and~\ref{thm:stability5}.

\begin{proof}[Proof of Theorem~\ref{thm:main23} and  proof of Theorem~\ref{thm:main45} in case that $G$ has a unique largest sum-free set]

We show the proof of Theorem~\ref{thm:main23} for $r=3$ only, the proof for $r=2$ and that
of Theorem~\ref{thm:main45}   in case that $G$ has a unique largest sum-free set follow the same line.

Let $G$ be a type~I($q$) group of sufficiently large order $n$
and suppose that $A \subset G$ maximizes the number of sum-free $3$-colorings. 
Then $\kappa_3(A)  \geq 3^{\mu(G)}$ as this bound is achieved by any largest sum-free set in $G$. Hence, we have $|A|\geq \mu(G)$.
Moreover, with $\epsilon =  \frac1{80(q+1)}$ Theorem~\ref{thm:stability23} implies that  there 
exists a largest sum-free set $B$ such that $|A \setminus B| < \epsilon n$.

Assume that there is an $x\in A\setminus B$ otherwise $A=B$ and the theorem follows. 
If $|G|$ is odd then by Lemma~\ref{lem:proof23exact} 
there are at least $\frac{n}{2q}-1$  pairs 
$\{a,b\}\in\binom{B}{2}$ such that $x=a+b$. 
If $|G|$ is even then by Lemma~\ref{lem:matching} (with $t=1$) for each $a\in B$ there exists $b\in B$ such that $x=a+b$, where possibly $a=b$.
If $\phi(x)=i \in[3]$ then $\phi$ can be extended  to at most 8 sum-free colorings of $\{a,b\}$ if $a\neq b$ and $a,b\in A$, to at most~2 
if $a=b$, $a\in A$ (hence $4<8$ ways for two such pairs  $a=b$ and $a'=b'$), and to at most 3 if $a\not\in A$ or $b\not\in A$.
Using $|A\setminus B|<\eps n$  and $\left(1-\frac12\log_38\right)>\frac{1}{20}$ we conclude therefore
\begin{align*}
\kappa_3(A)<  3^{\eps n}\cdot8^{\frac{n}{2q}} \cdot3^{\left(|A|-\frac n{q}+\epsilon n+2\right)}< 3^{\mu(G)+4\eps n-\frac n{20q}}=3^{\mu(G)}.
\end{align*}
 This contradicts the lower bound. Thus, $A= B$ and the theorem follows.
\end{proof}

\begin{proof}[Proof of Theorem~\ref{thm:main45}] It is left to prove Theorem~\ref{thm:main45} in the case that $G$ has at least two largest sum-free sets.
The proofs are quite involved but similar for $r=4$ and for $r=5$. Therefore we show the case $r=5$ only.

Suppose that $A\subset G$ maximizes the number of sum-free $5$-colorings.
Let $\gamma =3(\log n)^{-1/27}$ and let $\mathcal{F}$ be a container family of $A$ as in Theorem~\ref{thm:container}.
If $A$ maximizes the number of sum-free $5$-colorings, then 
$|\kappa_5(A)| \geq 6^{n/2}$ due to Proposition~\ref{prop:lowerbounds}. 
Let $(F_1,\dots, F_5)$ be the quintuple maximizing $|\Phi(F'_1,\dots,F'_5)|$ over all $(F'_1,\dots, F'_5)\in\cF^5$.
From \eqref{eq:PhiInContainer1} of Observation~\ref{obs:PhiInContainer}
we infer $6^{n/2}\leq\kappa_5(A)\leq |\cF|^5|\Phi(F_1,\dots, F_5)|$ and~(\ref{it:container1}) of Theorem~\ref{thm:container}  implies
\begin{align}\label{eq:substantial5}\log_6|\Phi(F_1,\dots,F_5)|\geq {\frac n2-2n(\log n)^{-1/18}}.\end{align}
Hence, we are in the position to apply  Theorem~\ref{thm:stability5} to $(F_1,\dots, F_5)$ which entails that there are largest sum-free sets $B_1,B_2,B_3$ in $G$
which satisfy the following.
\begin{enumerate}
 \item \label{it:stability1}  $B_3= B_1 \bigtriangleup B_2$ and  there is a function $f:[5]\to[3]$ such that 
$|F_j\setminus B_{f(j)}|<\gamma n$ for all $j\in[5]$ and
$\cB=(B_{f(1)},\dots, B_{f(5)})$ consists of one $4$-atom, and two $3$-atoms, or
 \item\label{it:stability2} there are four distinct largest sum-free sets $B_4,\dots, B_7$ contained in 
 $B_1\cup B_2\cup B_3$ and a function $f:[5]\to [7]$ such that 
$|F_j\setminus B_{f(j)}|<\gamma n$ for all $j\in[5]$ and
$\cB=(B_{f(1)},\dots, B_{f(5)})$ consists of two $2$-atoms, four $3$-atoms, and one $4$-atom. 
\end{enumerate}

\begin{claim}\label{claim:AcontU}
If \eqref{it:stability1} applies then $A\subseteq B_1\cup B_2$ and if \eqref{it:stability2} applies then we have $A\subseteq B_1\cup B_2\cup B_3$.
\end{claim}
\begin{proof}
The argument is similar to the one in the proof of Theorem~\ref{thm:main23}, i.e., we show that the existence of an $x\in A\setminus (B_1\cup B_2)$ 
or $x\in A\setminus (B_1\cup B_2\cup B_3)$, respectively,  gives a substantial restriction on how the elements of $A$ may be colored. This then leads to a contradiction to the bound on $\kappa_5(A)$.
As the proofs of these two cases are similar, we show the second case only.

 Assume that there is an $x\in A \setminus (B_1\cup B_2 \cup B_3)$ of which there are at most $5\gamma n$. As $B_{f(i)}\in \{B_1,B_2,B_3\}$ for all $i\in[5]$ the atoms
of $\cB$ and of $\cB'=(B_1,B_2,B_3)$ are the same and each  has size $n/8$ due to Corollary~\ref{remark:reduction}.
Suppose  $\phi(x)=j \in [5]$ so that $x\in F_j$ and let $\cB(\veceps)$ be a $3$-atom contained in $B_{f(j)}$. This
indeed exists as $B_{f(j)}$ consists of four atoms and $\mathcal{B}$ consists of  seven non-zero atoms,  
four of which are 3-atoms. 
By applying Lemma~\ref{lem:matching} to  $\mathcal{B}'$ we infer that
for each $a\in \cB(\veceps)$ there is a $b\in \cB(\veceps)$ such that $x=a+b$.
Hence $\phi$ can be extended  to at most 8 sum-free colorings of $\{a,b\}$ if $a\neq b$ and $a,b\in F_j$, to at most~2 
if $a=b$, $a\in F_j$ (hence $4<8$ ways for two such pairs  $a=b$ and $a'=b'$), and at most 3 if $a\not\in F_j$ or $b\not\in F_j$.
Together with Observation~\ref{obs:PhiInContainer} and the atom structure of $\cB$ this yields 
\begin{align} \label{eq:upbound}
|\Phi(F_1, \ldots, F_5)| \leq  5^{5\gamma n}\cdot 2^{n/4} \cdot 3^{3n/8} \cdot 4^{n/8} \cdot 8^{n/16}= 5^{5\gamma n}\cdot 6^{n/2}\cdot\left(\frac{\sqrt 8}3\right)^{n/8},
\end{align}
which is  a contradiction to \eqref{eq:substantial5}.
\end{proof}
We next establish that $B_1\cup B_2\subset A$ in the first case and $B_1\cup B_2\cup B_3\subset A$ in the second case. 
The following distinction of the colorings will be crucial for this purpose.
 \begin{definition}\label{def:goodbadcol}
A coloring $\varphi \in \Phi_5(A)$ is called \emph{good}  
 if for each $i\in[5]$ we have $\phi^{-1}(i)\subset L_i$ for some largest sum-free set $L_i$ of $G$, and
\begin{itemize}
\item the tuple $\mathcal{L}=(L_1,\dots,L_5)$ consists of one $4$-atom and two $3$-atoms, or
\item the tuple $\mathcal{L}=(L_1,\dots,L_5)$ consists
of two 2-atoms, four 3-atoms and one 4-atom.
\end{itemize}
 Otherwise $\phi$ is called a \emph{bad} coloring.
\end{definition}
Note that Definition~\ref{def:goodbadcol} 
distinguishes the two cases \eqref{it:stability1} and \eqref{it:stability2} 
as the number of atoms in $(L_1,\dots,L_5)$ is three in the first case and seven in the second case. 

Assume for contradiction that $x\in (B_1\cup B_2)\setminus A$ 
or $x\in (B_1\cup B_2\cup B_3)\setminus A$, respectively.
Crucial about the definition is that  any good coloring $\phi$ of $A$ can be extended to at least two sum-free $5$-colorings  of $A\cup\{x\}$.
Indeed, as $x$ belongs to some $k$-atom $\mathcal{L}(\veceps)$, $\veceps\in\{0,1\}^5$, of 
$\mathcal{L}=(L_1,\ldots,L_5)$ for $k\geq 2$,
the coloring  $\varphi$ can be extended 
to colorings of $A\cup\{x\}$ by assigning to $x$ one of the $k$ colors associated with $\mathcal{L}(\veceps)$.   
As $\phi^{-1}(i)\subset L_i$ for all $i\in[5]$
these  extensions are  sum-free.
Therefore, assuming the existence of  $x$ implies
\[|\Phi_5(A\cup\{x\})|\geq 2 |\{ \varphi \in \Phi_5(A): \varphi \text{ is good}\}|.\]

As good and bad colorings of $A$ partition $\Phi_5(A)$ it is sufficient to show the following.
\begin{claim}\label{claim:goodbadcol}
For $n$ sufficiently large
\[|\{ \varphi \in \Phi_5(A): \varphi \text{ is good}\}|> 6^{n/2}2^{-24\gamma n} >  \left(3\cdot2^{7/2}+\frac1{50}\right)^{n/4}>|\{ \varphi \in \Phi_5(A): \varphi \text{ is bad}\}|.\] 
In particular, if \eqref{it:stability1} holds then $B_1\cup B_2\subseteq A$ and if 
\eqref{it:stability2} holds then $B_1\cup B_2\cup B_3\subseteq A$.
\end{claim}

\begin{proof} 
As the proofs for the two cases are similar we present the proof for the second case only, i.e., when $A \subset B_1\cup B_2\cup B_3$.
Recall that $|F_i\setminus B_{f(i)}|<\gamma n$ and $\cB=(B_{f(1)},\dots, B_{f(5)})$ consists of two $2$-atoms, four $3$-atoms and one $4$-atom, each of which has size $n/8$. We deduce  that
$\sum_{i\in[5]}|B_{f(i)}\setminus F_i|<12 \gamma n$  since  the following contradiction to \eqref{eq:substantial5} would arise otherwise:
\[6^{n/2-2n(\log n)^{-1/18}}\leq |\Phi(F_1,\dots,F_5)|< 5^{5\gamma n}2^{n/4-12\gamma n}3^{n/2}4^{n/8}<6^{n/2-\gamma n}.\]
Due to the atom structure of $\cB$ note  that  $\phi\in\Phi_5(A)$ is  
good if  $\phi^{-1}(i)\subset  F_i\cap B_{f(i)}$ holds for all $i\in[5]$. Hence
 \begin{align}
 \label{eq:boundgood}
|\{ \varphi \in \Phi_5(A): \varphi \text{ is good}\}| > 2^{n/4}3^{n/2} 4^{n/8-12 \gamma n}=6^{n/2}2^{-24\gamma n}.
 \end{align}
 
 We now turn to bound the number of bad colorings of $A$. 
 Suppose that $(C_1,\dots, C_5)\in\cF^5$ maximizes $|\{ \varphi \in \Phi(F'_1, \dots, F'_5): \varphi \text{ is bad}\}|$ over all $(F'_1, \dots, F'_5)\in\cF^5$.
By~\eqref{eq:PhiInContainer1} of Observation~\ref{obs:PhiInContainer} and $\log_2|\cF|\leq n(\log n)^{-1/18}$ 
due to~\eqref{it:container1} of Theorem~\ref{thm:container} we have
\begin{align}
\begin{split}
\label{eq:boundbad}
|\{ \varphi \in \Phi_5(A): \varphi \text{ is bad}\}|& \leq\sum_{(F_1,\dots,F_5)\in\cF^5} \left|\{ \varphi \in \Phi(F_1, \ldots, F_5): \varphi \text{ is bad}\}\right|\\
&\leq  2^{5n(\log n)^{-1/18}} \left|\{ \varphi \in \Phi(C_1, \ldots, C_5): \varphi \text{ is bad}\}\right|.
\end{split}
\end{align}

Therefore suppose that $|\Phi(C_1,\dots, C_5)|\geq \left(3\cdot2^{7/2}+\frac1{100}\right)^{n/4}$ otherwise the claim follows. Theorem~\ref{thm:stability5} 
then applies and there are largest sum-free sets $L_1, L_2, L_3$ and  a $g:[5]\to [3]$ or $g:[5]\to [7]$
such that \eqref{it:stability51} or \eqref{it:stability52} of Theorem~\ref{thm:stability5} holds with $F_i,B_i, f,\cB$ replaced by  $C_i,L_i,g$ and $\cL=(L_{g(1)},\dots, L_{g(5)})$.
As $C_1\cup\dots \cup C_5=A=F_1\cup \dots\cup F_5$, the sets 
$L_1,L_2,L_3$ must be contained in $B_1\cup B_2\cup B_3$ (i.e., $L_i\in\{B_1,\dots ,B_7\}$) 
and no $L_i$, $i\in[3]$, is contained in the union of the other two.
Hence, case~\eqref{it:stability52} of Theorem~\ref{thm:stability5} applies to $(C_1, \ldots, C_5)$ and  
$\cL$  consists of two $2$-atoms, four $3$-atoms and  one $4$-atom.
In particular, a $\phi \in \Phi(C_1, \ldots,C_5)$ satisfying $\phi^{-1}(i) \subset L_{g(i)}$ for all $i \in [5]$ is a good
coloring and   a bad 
$\phi\in\Phi(C_1, \dots, C_5)$  must  therefore exhibit a $k\in[5]$ and an  $x\in C_k\setminus  L_{g(k)}$ such that $\phi(x)= k$. 
Call  $(x,k)$ a bad pair and for such  pair let $\Phi(C_1, \dots, C_5|x \mapsto k)$ be 
the set of all (bad)   $\varphi \in \Phi(C_1, \dots, C_5)$  with $\varphi(x) =k$. Suppose $(y,\ell)$ maximizes $|\Phi(C_1, \dots, C_5|x \mapsto k)|$ over all bad pairs $(x,k)$, then
an argument as in the proof of Claim~\ref{claim:AcontU} will show that  $|\Phi(C_1, \dots, C_5|y \mapsto \ell)|$ is small. 

Indeed, fix a $3$-atom $\cL(\veceps)$ in $L_{g(\ell)}$ which exists since $L_{g(\ell)}$ consists of four atoms and $\cL$ consists of seven atoms, four of which are $3$-atoms.
By Lemma~\ref{lem:matching}  for each $a\in \cL(\veceps)$ there is a $b\in \cL(\veceps)$ such that $y=a+b$.
Coloring $y$ with $\ell$ therefore extends to at most 8 sum-free colorings of $\{a,b\}$ if $a\neq b$ and $a,b\in C_\ell$, to at most~2 
if $a=b$, $a\in C_\ell$ (hence $4<8$ ways for two such pairs  $a=b$ and $a'=b'$), and to at most 3 if $a\not\in C_\ell$ or $b\not\in C_\ell$. Hence,
as  $|C_i\setminus L_{ g(i)}|<\gamma n$, $i\in[5]$, we have 
\begin{align*}
|\{ \varphi \in \Phi(C_1, \dots,C_5): \varphi \text{ is bad}
\}|  &\leq {5\gamma n} \cdot  |\Phi(C_1, \dots, C_5|y \mapsto \ell)|\\
&\leq {5\gamma n} \cdot 5^{5 \gamma n}   \cdot 2^{n/4} \cdot 3^{3n/8} \cdot 4^{n/8} \cdot 8^{n/16}\\
&= {5\gamma n} \cdot 5^{5\gamma n}\cdot 6^{n/2}\cdot\left(\frac{\sqrt 8}3\right)^{n/8}<5.93^{n/2}< \left(3\cdot2^{7/2}+\frac1{100}\right)^{n/4}
\end{align*}
for  $n$ sufficiently large. In view of \eqref{eq:boundbad} the claim follows.
\end{proof}

We have established that  
either $A=B_1\cup B_2$  for an independent $(B_1,B_2)$ or $A=B_1\cup B_2\cup B_3$ for an independent $(B_1,B_2,B_3)$.
Moreover, the proof shows that the number of colorings in $\Phi_5(A)$ is largely dominated by the good colorings.

Call a quintuple $\cD=(D_1,\dots, D_5)$ of  largest sum-free sets in $A$ 
\emph{substantial} if $\cD$ consists of $2$-atoms, four $3$-atoms and one $4$-atom. By definition a coloring $\phi\in\Phi_5(A)$ is good if and only if 
$\phi\in \Phi(\cD)$ for some substantial $\cD$. Further, if $\cD$ and $\cD'$ are two distinct substantial quintuple, then $|\Phi(\cD)\cap \Phi(\cD')|=o(|\Phi(\cD)|)$, i.e., most of the 
good colorings of $A$ are assigned to a substantial $\cD$ in a unique way.
Indeed,  a $\phi\in\Phi(\cD)\setminus\Phi(\cD')$  
if for every $k$-atom $\cD(\veceps)$, $k\in\{2,3,4\}$, all $k$ colors are represented in $\cD(\veceps)$ under $\phi$, i.e., $|\phi(\cD(\veceps))|=k$. It is easily seen that most $\phi\in\Phi(\cD)$
satisfy this property.

Finally note that a  substantial tuple $(D_1,\dots,D_5)$ in $A=B_1\cup B_2$ are  exactly those having $B_1,B_2$ and $B_3$ as members, two of them twice.
There are $\binom{3}2\binom523=90$ many ways to choose such a tuple.
Similarly, a  substantial tuple $(D_1,\dots,D_5)$ in $A=B_1\cup B_2\cup B_3$ must contain an independent triple, say $\cB'=(B_1',B_2',B_3')$, 
for which there are $\binom72\cdot4$ many choices.
These either extend to a quintuple consisting of pairwise distinct largest sum-free sets, all of which are substantial, 
 or to a quintuple in which exactly one member is repeated, from which all but those containing $\cB'(1,1,1)$ are substantial. 
 No other tuple is substantial. Together this yields $\binom72\cdot4\left(4\cdot3\cdot 5!+3\cdot4\cdot \frac{5!}2\right)=181440$
ways to choose a substantial tuple in $A=B_1\cup B_2\cup B_3$. This finishes the proof.
\end{proof}

\section{Concluding remarks}
\label{sec:concludingremarks}
In the case that $G$ is an even order abelian group the proof of  Theorem~\ref{thm:main45} consists of two parts, a stability and the exact part.
Both parts can be extended to more colors provided $G$ contains   sufficiently many
 largest sum-free sets. 
Indeed, for $r=4,5$ the stability results, Theorem~\ref{thm:stability4} and~\ref{thm:stability5}, follow from  Lemma~\ref{lem:alllargest} and~\ref{lem:OptStructure}.
The short argument in Lemma~\ref{lem:alllargest} extends  to $r=6$ without
change using the lower bound $\kappa_{6,G}\geq 2^{3n/4}3^{n/2}$ (see the footnote in page~\pageref{eq:allgood}). 
For $r=7$ we have the lower bound $\kappa_{7,G}\geq\kappa_7(B_1\cup B_2\cup B_3)\geq 4^{7n/8}$ obtained by an independent triple $(B_1,B_2,B_3)$
of largest sum-free sets in $G$. The same argument as in Lemma~\ref{lem:alllargest}  then implies that for $|\Phi(F_1,\dots, F_7)|\geq 3.99^{7n/8}$ to hold 
 at least six of the seven sets from $(F_1,\dots, F_7)$ must be good and it can then be easily argued  that the last $F_i$ must be good as well.
The second ingredient for the stability result is Lemma~\ref{lem:OptStructure} which identifies
the tuples $\cB=(B_1,\dots, B_r)$ maximizing $|\Phi(\cB')|$ over all tuple $\cB'$ of largest sum-free sets. 
This lemma can be extended to $r=6, 7$ along the same line. Alternatively, as Corollary~\ref{remark:reduction} reduces 
the problem in Lemma~\ref{lem:OptStructure} for an arbitrary $G$ to a related one in $\mathbb F_2^t$ for $t\leq r$, 
the problem might also be  solved by computer search, if needed.

It can be verified that for $r=6$ the optimal structure is obtained for $t=3$ and $\cB$ consisting of four 3-atoms and three 4-atoms,
and for $r=7$ it is obtained for $t=3$ and $\cB$ consisting of  seven  4-atoms.
The optimal  structure  for an arbitrary $r$ is unknown.

Given the stability result  the exact configuration can then be derived using  the  argument given in Section~\ref{sec:exact}
which works for all fixed $r$. Together, we obtain the following for $r=6,7$.

\begin{theorem}
Let $r\in\{6,7\}$ and let $G$ be an abelian group  of sufficiently large even order. 
Then $\kappa_r(A)=\kappa_{r,G}$ if and only if $A=B_1\cup B_2\cup B_3$ for an independent triple $(B_1,B_2,B_3)$ of largest sum-free sets, provided such exists in $G$.
\end{theorem}

It would be interesting to extend the results concerning even order groups  to arbitrary $r$ and from even order groups to arbitrary  type I($q$) groups (for $r\geq 4$).
The methods presented here might be pushed further to give an answer to certain cases of $r>7$ and even order groups. They
might also be adapted give an answer to particular cases of type I($q$) groups and $r\geq 4$. However, new ideas are needed to solve these problems in general.
\bibliographystyle{plain}
\bibliography{biblio_sumfree}

\end{document}